\documentclass{amsart}

\usepackage{microtype}
\usepackage{amscd}
\usepackage{amsmath,amssymb,amsfonts}
\usepackage{xypic}
\usepackage{xstring}
\usepackage[mathscr]{euscript}
\usepackage[cmtip,all,2cell]{xy}
\usepackage{url}
\usepackage{latexsym}
\usepackage{amsthm}
\usepackage[latin1]{inputenc}
\usepackage{multicol}
\usepackage{enumitem}
\usepackage{hyperref}
\usepackage{tikz}
\usepackage{color}

\UseAllTwocells

\theoremstyle{plain}
\newtheorem{theorem}{Theorem}[section]

\newtheorem{lemma}[theorem]{Lemma}

\newtheorem{proposition}[theorem]{Proposition}

\newtheorem{corollary}[theorem]{Corollary}

\theoremstyle{definition}

\newtheorem{definition}[theorem]{Definition}
\newtheorem{variant}[theorem]{Variant}
\newtheorem{construction}[theorem]{Construction}
\newtheorem{example}[theorem]{Example}

\newtheorem{observation[theorem]}{Observation}
\newtheorem{remark}[theorem]{Remark}

\newcommand{\mf}[1]{\mathfrak{#1}}
\newcommand{\mc}[1]{\mathcal{#1}}

\newcommand{\ol}[1]{\overline{#1}}
\newcommand{\cat}[1]{
\StrLen{#1}[\mystrlen]
\ifnum\mystrlen=1 \mathscr{#1}
\else \mathrm{#1}
\fi}

\newcommand{\oocat}[1]{
\StrLen{#1}[\mystrleng]
\ifnum\mystrleng=1 \mathscr{#1}
\else \mathrm{#1}
\fi
}

\newcommand{\Set}[0]{\oocat{Set}}
\newcommand{\sSet}[0]{\oocat{sSet}}
\newcommand{\sS}[0]{\cat{S}}

\newcommand{\dau}[0]{\partial}

\newcommand{\mb}[1]{\mathbf{#1}}
\newcommand{\ho}[0]{\mathrm{ho}}
\newcommand{\mm}[1]{\mathrm{#1}}

\newcommand{\Map}[0]{\mm{Map}}
\renewcommand{\rto}[2][\;]{\xrightarrow[#1]{#2}}

\newcommand{\rt}[0]{\rightarrow}
\newcommand{\lt}[0]{\leftarrow}

\newcommand{\Fun}[0]{\cat{Fun}}

\title{Localizing $\infty$-categories with hypercovers}
\author{Joost Nuiten}

\begin{document}

\begin{abstract}
Given an $\infty$-category with a set of weak equivalences which is stable under pullback, we show that the mapping spaces of the corresponding localization can be described as group completions of $\infty$-categories of spans. Furthermore, we show how these $\infty$-categories of spans are the mapping objects of an $(\infty, 2)$-category, which yields a Segal space model for the localization after a Kan fibrant replacement. 
\end{abstract}

\maketitle

\vspace{-7pt}
\section{Introduction}
The notion of a relative category, i.e.\ a category $\cat{C}$ equipped with a subcategory $W$ of weak equivalences, is ubiquitous in homotopy theory: from an early stage on, essentially every homotopy theory has been described by a relative category, which often serves as a better behaved placeholder for its homotopy category $\ho(\cat{C})$, the category obtained by formally turning the weak equivalences into isomorphisms (see e.g.\ \cite{gab67}).

It is nowadays understood that there is a more suitable and richer object that sits in between the category $\cat{C}$ and its homotopy category $\ho(\cat{C})$, given by the $\infty$-category $\cat{C}[W^{-1}]$ obtained by formally turning the weak equivalences in $\cat{C}$ into \emph{homotopy equivalences}. When $\infty$-categories are interpreted as simplicially enriched categories, the simplicial category $\cat{C}[W^{-1}]$ can be described (for example) by the simplicial localization of Dwyer and Kan \cite{dwy80b}. The $\infty$-category $\cat{C}[W^{-1}]$ usually reflects the good properties that the relative category $(\cat{C}, W)$ has; for example, when $(\cat{C}, W)$ is a combinatorial model category, the $\infty$-category $\cat{C}[W^{-1}]$ is locally presentable.

Although the point-set model for $\cat{C}[W^{-1}]$ provided by the Dwyer-Kan simplicial localization (or the closely related hammock localization) is useful for formal applications, it tends to be rather involved to understand its structure concretely. For example, the mapping spaces in the hammock localization consist of arbitrary zig-zags of maps and weak equivalences in $\cat{C}$. The mapping spaces in $\cat{C}[W^{-1}]$ have a much simpler description when the weak equivalences in $\cat{C}$ are stable under base change (a fact we emphasize by calling them \emph{hypercovers}): in this case one can reduce the length of the hammocks and describe the mapping spaces as nerves of `cocycle categories' whose objects are spans $\cdot \stackrel{\sim}{\lt}\cdot \rt \cdot$ whose left arrow is a hypercover (see e.g.\ \cite{cis10b, jar06}, and \cite{hor15} for the (homotopy coherent) composition). 

A classical situation in which one wants to invert such a class of hypercovers arises in the study of groupoid objects: for example, the category of Lie groupoids (i.e.\ groupoids internal to smooth manifolds) comes equipped with a class of weak equivalences known as Morita equivalences (the natural smooth analogue of an equivalence of groupoids) and formally inverting these equivalences yields a model for the 2-category of differentiable stacks (see e.g. \cite{pro96}). Although the Morita equivalences themselves are not stable under base change, they are \emph{generated} under the 2-out-of-3 property by the hypercovers (fully faithful maps that are surjective submersions on objects), which are stable under base change. One can therefore equivalently localize at these hypercovers, and obtain a description of the 2-category of differentiable stacks in terms of Lie groupoids and spans between them (with one map being a hypercover). Natural analogues of Lie groupoids also appear in derived-geometric settings: for example, in derived algebraic geometry one naturally encounters groupoid objects internal to derived schemes (see e.g.\ \cite{pri09}), which \emph{themselves} form an $\infty$-category. To study the localization of such an $\infty$-category of derived-geometric groupoids at its hypercovers (or equivalently, at its Morita equivalences), one would like an analogous description of the resulting mapping spaces in terms of $\infty$-categories of spans $\cdot \stackrel{\sim}{\lt}\cdot \rt \cdot$.

There are many other situations in which one may want to invert a set of `external' weak equivalences in an $\infty$-category, which itself already comes with a natural `internal' notion of homotopy equivalence. For example, when studying the $\infty$-category of topological monoids, one may consider the $\infty$-category of simplicial diagrams of such topological monoids, with the trivial Kan fibrations between them as the external weak equivalences (see also Example \ref{ex:algebras}). This provides a way to work with simplicial resolutions of topological monoids (for example, the free resolution provided by the bar construction) without having to deal with point-set issues arising in the usual homotopy theory of spaces.

The aim of this paper is to provide a model for the localization $\cat{C}[W^{-1}]$ of an $\infty$-category $\cat{C}$ at a set of hypercovers, in terms of cocycle $\infty$-categories. More precisely, we will show that associated to any $\infty$-category with hypercovers $(\cat{C}, W)$ is a certain $(\infty, 2)$-category, whose mapping objects are given by $\infty$-categories of spans $\cdot \stackrel{\sim}{\lt}\cdot \rt \cdot$ whose left arrows are hypercovers. The $\infty$-category obtained by group completing each of these mapping categories (i.e. by taking a Kan fibrant replacement) then gives a model for the localization $\cat{C}[W^{-1}]$ (see Theorem \ref{thm:bicatofspans}); in particular, the mapping spaces of $\cat{C}[W^{-1}]$ arise as group completions of $\infty$-categories of spans.

The paper is organized as follows: in Section \ref{sec:localization} we show how the mapping spaces of an $\infty$-categorical localization $\cat{C}[W^{-1}]$ can be described in terms of left fibrations satisfying some universal property. This allows one to study the mapping spaces of $\cat{C}[W^{-1}]$ without having to deal with the composition maps between them. Using this, we will show in Section \ref{sec:localizinghypercovers} how the localization of an $\infty$-category with hypercovers has mapping spaces given by group completions of $\infty$-categories of spans $\cdot \stackrel{\sim}{\lt}\cdot \rt \cdot$ whose left arrow is a hypercover. In Section \ref{sec:2catoffractions} we show how these $\infty$-categories of spans form the mapping categories of an $(\infty, 2)$-category, and how group-completing each of these mapping categories yields a model for $\cat{C}[W^{-1}]$.

\subsection*{Conventions} Throughout, we will use quasicategories as our chosen model for $\infty$-categories. We will use $\sS$ to denote the $\infty$-category of spaces (i.e.\ the coherent nerve of the simplicial category of Kan complexes), $\cat{C}/c$ to denote the usual over-$\infty$-category of $\cat{C}$ in the sense of Joyal and $\cat{C}^{/c}$ to denote the `alternative slice construction' of \cite[Section 4.2.1]{lur09}.

\subsection*{Acknowledgements} The author was supported by NWO.

\section{Mapping spaces in localized $\infty$-categories}\label{sec:localization}
Let $\cat{C}$ be an $\infty$-category equipped with a set of maps $W$ which is closed under homotopy, composition, base change and which contains all the equivalences. The aim of this section is to provide a simple (model-categorical) description of the mapping spaces of the localization $\cat{C}[W^{-1}]$ (see Corollary \ref{cor:mappingspacesfromfibrantreplacement}) in terms of left fibrations. This manoeuvre can be used to describe the mapping spaces of $\cat{C}[W^{-1}]$ without having to specify (coherently associative) composition maps between them. In particular, the flexibility of changing to weakly equivalent left fibrations allows us to provide various descriptions of the mapping spaces of $\cat{C}[W^{-1}]$ (e.g.\ a description in terms of spans, as in Section \ref{sec:localizinghypercovers}).

\begin{definition}\label{def:localization}
A relative $\infty$-category $(\cat{C}, W)$ is a (small) quasicategory $\cat{C}$, together with a set $W\subseteq \cat{C}_1$ of arrows in $\cat{C}$ which is closed under homotopy, composition and contains all equivalences in $\cat{C}$ (such $W$ is also called a \emph{system} in \cite{lur11}).

If $(\cat{C}, W)$ is a relative $\infty$-category, its \emph{localization} is a functor $u\colon \cat{C}\rt \cat{C}[W^{-1}]$ with the universal property that for any $\infty$-category $\cat{D}$, the map of spaces
$$\xymatrix{
u^*\colon \mm{Map}(\cat{C}[W^{-1}], \cat{D})\ar[r] & \mm{Map}(\cat{C}, \cat{D})
}$$
is the inclusion of the connected components consisting of those functors $f\colon \cat{C}\rt \cat{D}$ that send maps in $W$ to equivalences in $\cat{D}$.
\end{definition}
\begin{remark}\label{rem:invarianceunder2-3}
It follows immediately from the universal property that the localization at a set of maps $W$ is equivalent to the localization at the set of maps $\ol{W}$ obtained by closing $W$ under the $2$-out-of-$3$ property in the homotopy category $\ho(\cat{C})$. For example, this can be used to show that the localization of a category of fibrant objects (in the sense of Brown \cite{bro73}) at the weak equivalences is equivalent to its localization at the trivial fibrations. The latter are stable under base change, which allows for an easier description of the localization (see Section \ref{sec:localizinghypercovers}).
\end{remark}
\begin{remark}\label{rem:localizationaspushout}
Any (small) relative $\infty$-category $(\cat{C}, W)$ has a localization; in terms of simplicial categories, this is the Dwyer-Kan simplicial localization \cite{dwy80b}. In terms of quasicategories, let us abuse notation and denote by $\cat{W}\subseteq \cat{C}$ the maximal sub-simplicial set of $\cat{C}$ whose $1$-simplices are given by the set $W$. Since $W$ is closed under composition and contains all identity maps, $\cat{W}\subseteq \cat{C}$ is a quasicategory with the same set of objects as $\cat{C}$. Now let $\cat{W}\rt \cat{W}[W^{-1}]$ be the \emph{group completion} of $\cat{W}$, i.e.\ a fibrant replacement of $\cat{W}$ in the Kan-Quillen model structure, and let $\cat{C}[W^{-1}]$ be a Joyal-fibrant replacement of the (homotopy) pushout
\begin{equation}\label{diag:localizationaspushout}\vcenter{\xymatrix{
\cat{W}\ar[d]\ar[r] & \cat{W}[W^{-1}]\ar[d]\\
\cat{C}\ar[r]_-u & \cat{C}[W^{-1}].
}}\end{equation}
The functor $u\colon \cat{C}\rt \cat{C}[W^{-1}]$ realizes the localization of $\cat{C}$ at $W$. Since the Joyal and Kan-Quillen model structures admit fibrant replacement functors that are the identity on vertices, we may always assume that the functor $u\colon \cat{C}\rt \cat{C}[W^{-1}]$ is the identity map on vertices. Because of this, we tend to identify the objects of $\cat{C}$ and $\cat{C}[W^{-1}]$.
\end{remark}
To analyze the mapping spaces of $\cat{C}[W^{-1}]$, we will use the following simple observation: for any object $c\in \cat{C}$ (or in $\cat{C}[W^{-1}]$, by the above remark), there is a space-valued functor $\cat{C}[W^{-1}](c, -)\colon \cat{C}[W^{-1}]\rt \sS$. By the universal property of the functor $u\colon \cat{C}\rt \cat{C}[W^{-1}]$, it suffices to understand the values of the functor $u^*\cat{C}[W^{-1}](c, -)\colon \cat{C}\rt \sS$. We will show that the latter functor is itself characterized by a universal property in $\mm{Fun}(\cat{C}, \sS)$: it is the universal functor under the representable $\cat{C}(c, -)$ that sends maps in $W$ to equivalences of spaces. To this end, we will describe the situation in terms of model categories as follows:
\begin{definition}
Let $(\cat{C}, W)$ be a relative quasicategory. The \emph{$W$-local} covariant model structure, which we will denote by $\sSet^\mm{cov}/(\cat{C}, W)$, is the left Bousfield localization of the covariant model structure on $\sSet/\cat{C}$ at all maps
\begin{equation}\label{eq:localizingmaps}\vcenter{\xymatrix@R=1.3pc{
\{1\}\ar[rr] \ar[rd] & & \Delta[1]\ar[ld]^w\\
& \cat{C} &
}}\end{equation}
where $w\colon \Delta[1]\rt \cat{C}$ is a map in $W$. We will refer to the weak equivalences of this model structure as \emph{$W$-local equivalences} and to the fibrant objects of this model structure as \emph{$W$-local left fibrations}.
\end{definition}
\begin{lemma}\label{lem:characterizingWlocalfibrations}
Let $p\colon X\rt \cat{C}$ be a left fibration over a relative category $(\cat{C}, W)$. Then the following are equivalent:
\begin{enumerate}
 \item $p$ is a $W$-local left fibration, i.e.\ $p$ is fibrant in $\sSet^\mm{cov}/(\cat{C}, W)$.
 \item the base change $p|_{\cat{W}}\colon \cat{W}\times_{\cat{C}} X\rt \cat{W}$ along the inclusion $\cat{W}\rt \cat{C}$ (Remark \ref{rem:localizationaspushout}) is a Kan fibration.
 \item for every $w\colon \Delta[1]\rt \cat{C}$ in $W$, the base change $p'\colon w^*X\rt \Delta[1]$ is a Kan fibration.
\end{enumerate}
\end{lemma}
\begin{proof}
It is clear that (2) implies (3). To see that (1) implies (2), it suffices to prove that any map
$$\xymatrix@R=1.3pc{
\Lambda^n[n]\ar[rr]\ar[rd] & & \Delta[n]\ar[ld]^w\\
 & \cat{C}
}$$
with $w$ taking values in $\cat{W}\subseteq \cat{C}$ is a trivial cofibration in the $W$-local covariant model structure. This follows from an easy inductive argument on $n$: when $n=1$ the map is one of the localizing maps. If all right horn inclusions over $\cat{W}$ are $W$-local trivial cofibrations in dimensions $<n$, then the map $\{n\}\rt \Lambda^n[n]$ is a $W$-local trivial cofibration, being an iterated pushout of such. It therefore suffices to prove that the map $\{n\}\rt \Delta[n]$ is a $W$-local trivial cofibration. But this map factors as $\{n\}\rt \mm{Sp}[n]\rt \Delta[n]$, where the last map is the spine inclusion (hence inner anodyne) and the first map is an iterated pushout of the localizing maps \eqref{eq:localizingmaps}.

Finally, to see that (3) implies (1), suppose that the restriction of $p$ to each $w\colon \Delta[1]\rt \cat{W}\subseteq \cat{C}$ is a Kan fibration. This implies that $p$ has the right lifting property with respect to each pushout-product map
$$\xymatrix@R=1.2pc{
\{1\}\times L\coprod_{\{1\}\times K} \Delta[1]\times K\ar[rr]^-{j}\ar[rd] & & \Delta[1]\times L\ar[ld]^{\pi_1}\\
& \Delta[1]\ar[d]^w & \\
& \cat{C} 
}$$
with $w$ taking values in $\cat{W}$, since the map $j$ is anodyne. But this means precisely that the left fibration $p\colon X\rt \cat{C}$ is local with respect to the localizing map \eqref{eq:localizingmaps}.
\end{proof}
\begin{remark}\label{rem:kanfibrationover1}
Let $X\rt \Delta[1]$ be a left fibration with fibers $X_0, X_1$ over $0$ and $1$. There exists a diagonal lift in the diagram
$$\xymatrix{
X_0\times \{0\}\ar[r]\ar[d] & X\ar[d]\\
X_0\times \Delta[1]\ar@{..>}[ru]\ar[r]_-{\pi} & \Delta[1]
}$$
whose restriction to $X_0\times\{1\}$ yields a map $X_0\rt X_1$, unique up to homotopy. This map $X_0\rt X_1$ is a weak equivalence of Kan complexes if and only if the dotted diagonal map is a fiberwise equivalence from a Kan fibration to a left fibration. In turn, this is equivalent to $X\rt \Delta[1]$ being a Kan fibration.
\end{remark}
\begin{proposition}
Let $(\cat{C}, W)$ be a relative $\infty$-category and let $u\colon \cat{C}\rt \cat{C}[W^{-1}]$ be its localization. Composition with $u$ and base change along $u$ then determine a Quillen equivalence
\begin{equation}\label{eq:equivalencebetweenrelativefunctors}\xymatrix{
u_!\colon \sSet^\mm{cov}/(\cat{C}, W)\ar@<1ex>[r] & \sSet^\mm{cov}/\cat{C}[W^{-1}]\ar@<1ex>[l] \colon u^*.
}\end{equation}
\end{proposition}
\begin{proof}
Since any two models for $\cat{C}[W^{-1}]$ are categorically equivalent under $\cat{C}$ and because the covariant model structure is invariant under categorical weak equivalences (see \cite[Remark 2.1.4.11]{lur09} or \cite[Theorem F]{heu13}), it suffices to prove this in the case where $\cat{C}[W^{-1}]$ is modeled simply by the \emph{strict} pushout
$$\xymatrix{
\cat{W}\ar[r]^-{v}\ar[d]_i & \cat{W}[W^{-1}]\ar[d]^{i'}\\
\cat{C}\ar[r]_-u & \cat{C}[W^{-1}]
}$$
where $\cat{W}\rt \cat{W}[W^{-1}]$ is an anodyne map whose target is a Kan complex. In other words, we do not replace $\cat{C}[W^{-1}]$ by a categorically equivalent quasicategory. To see that the adjunction \eqref{eq:equivalencebetweenrelativefunctors} is a Quillen pair, it suffices to prove that for any left fibration $p\colon X\rt \cat{C}[W^{-1}]$, the base change $u^*(p)$ (which is a left fibration) restricts to a Kan fibration on $\cat{W}$. But this is obvious, since $i^*u^*(p)= v^*i'^*(p)$ and $i'^*(p)$ is a left fibration over a Kan complex, hence a Kan fibration.

To see that $u^*$ is a Quillen equivalence, recall from \cite{moe16} that there is a Quillen equivalence
\begin{equation}\label{eq:pullbackmodeldescent}\xymatrix{
\mm{colim}\colon \sSet^\mm{cov}/\cat{C}\times^h_{\sSet^\mm{cov}/\cat{W}} \sSet^\mm{cov}/\cat{W}[W^{-1}] \ar@<1ex>[r] & \sSet^\mm{cov}/\cat{C}[W^{-1}]\ar@<1ex>[l]\colon \mm{restrict}
}\end{equation}
between the covariant model structure on $\sSet/\cat{C}$ and the left hand \emph{homotopy pullback model structure}; the latter is the category of diagrams of simplicial sets $X_1\lt X_0\rt X_2$ over the diagram $\cat{C}\lt \cat{W}\rt \cat{W}[W^{-1}]$, endowed with the model structure where a trivial fibration is simply an objectwise trivial fibration of diagrams and whose fibrant objects are given by diagrams
$$\xymatrix{
X_1\ar[d] & X_0\ar[l]\ar[r]\ar[d] & X_2\ar[d]\\
\cat{C} & \cat{W}\ar[l]\ar[r] & \cat{W}[W^{-1}]
}$$
whose vertical arrows are left fibrations and such that the two maps from $X_0$ to the pullbacks are covariant weak equivalences. Note that in this situation, the map $X_2\rt \cat{W}[W^{-1}]$ is a left fibration over a Kan complex and therefore a Kan fibration. Since any left fibration which is weakly equivalent to a Kan fibration is itself a Kan fibration, the left fibration $X_0\rt \cat{W}$ is a Kan fibration and the left fibration $X_1\rt \cat{C}$ restricts to a Kan fibration on $\cat{W}$. 

It follows that the homotopy pullback model structure of the covariant model structures is actually the \emph{same} as the homotopy pullback model structure of the $W$-local model structures
$$
\sSet^\mm{cov}/(\cat{C}, W)\times^h_{\sSet^\mm{cov}/(\cat{W}, W)} \sSet^\mm{cov}/\cat{W}[W^{-1}].
$$
But the $W$-local model structure on $\sSet/\cat{W}$ and the covariant model structure on $\sSet/\cat{W}[W^{-1}]$ are simply the Kan-Quillen model structures, so that the Quillen pair
$$\xymatrix{
v_!\colon \sSet^\mm{cov}/(\cat{W}, W)\ar@<1ex>[r] & \sSet^\mm{cov}/\cat{W}[W^{-1}]\ar@<1ex>[l]\colon v^*
}$$
is a Quillen equivalence. This implies that the Quillen pair
\begin{equation}\label{eq:projectionofpullbackmodelstructure}\xymatrix{
\sSet^\mm{cov}/(\cat{C}, W) \ar@<1ex>[r] & \sSet^\mm{cov}/(\cat{C}, W)\times^h_{\sSet^\mm{cov}/(\cat{W}, W)} \sSet^\mm{cov}/\cat{W}[W^{-1}]\ar@<1ex>[l]
}\end{equation}
is a Quillen equivalence. But the original Quillen pair \eqref{eq:equivalencebetweenrelativefunctors} is the composite of the Quillen pairs \eqref{eq:pullbackmodeldescent} and \eqref{eq:projectionofpullbackmodelstructure}.
\end{proof}
\begin{variant}\label{var:operadic}
The above results apply not only to relative $\infty$-categories, but to relative $\infty$-operads as well. Indeed, suppose that $\cat{O}$ is an $\infty$-operad, modeled by a fibrant dendroidal set (see \cite{cis09}) and let $\cat{O}_{\langle 1\rangle}$ be its underlying quasicategory (i.e.\ $\cat{O}_{\langle 1\rangle}=i^*\cat{O}$ in the notation of \cite{cis09}). A relative $\infty$-operad $(\cat{O}, W)$ is an $\infty$-operad, together with a class of arrows making $(\cat{O}_{\langle 1\rangle}, W)$ a relative $\infty$-category. The localization $\cat{O}\rt \cat{O}[W^{-1}]$ is the universal $\infty$-operad under $\cat{O}$ in which the arrows from $W$ are inverted (i.e.\ invertible in the underlying $\infty$-category $\cat{O}[W^{-1}]_{\langle 1\rangle}$). It admits an explicit description like in Remark \ref{rem:localizationaspushout}, as the pushout of $\cat{W}\rt \cat{O}$ along the map from the quasicategory $\cat{W}$ to its Kan fibrant replacement $\cat{W}[W^{-1}]$.

Now consider the covariant model structure on $\cat{dSet}/\cat{O}[W^{-1}]$, which is invariant under replacing $\cat{O}[W^{-1}]$ by an equivalent dendroidal set by \cite[Theorem 6.8]{heu11}. The above argument asserts that this model category is Quillen equivalent to the $W$-local covariant model structure on $\cat{dSet}/\cat{O}$, given by left Bousfield localization at the maps
$$\xymatrix@R=1.2pc{
\{1\}\ar[rr]\ar[rd] & &  \Delta[1]\ar[ld]^w\\
& \cat{O}
}$$
where the map $w\colon \Delta[1]\rt \cat{O}$ is contained in $W$. Using that the covariant model structure on $\cat{dSet}/\cat{O}$ is a model for the $\infty$-category of $\cat{O}$-algebras in spaces (see \cite{heu11}), this result can be interpreted as follows: the $\infty$-category of $\cat{O}[W^{-1}]$-algebras (in spaces) is equivalent to the $\infty$-category of $\cat{O}$-algebras (in spaces) with the property that the unary operations in $W$ act on them by equivalences.
\end{variant}
\begin{corollary}\label{cor:fullyfaithfuloutoflocalization}
Let $(\cat{C}, W)$ be a relative $\infty$-category and let $f\colon \cat{C}\rt \cat{D}$ be a functor of $\infty$-categories sending the maps in $W$ to equivalences in $\cat{D}$. Then the following two statements are equivalent:
\begin{enumerate}
\item the induced functor $\cat{C}[W^{-1}]\rt \cat{D}$ is fully faithful.
\item for each object $c$, the square
\begin{equation}\label{diag:mappingspaces}\vcenter{\xymatrix{
c/\cat{C}\ar[r]^f\ar[d] & f(c)/\cat{D}\ar[d]\\
\cat{C}\ar[r]_f & \cat{D}
}}\end{equation}
induces a $W$-local equivalence of left fibrations $c/\cat{C}\rt f^*\big(f(c)/\cat{D}\big)$ over $\cat{C}$, whose target is a $W$-local left fibration.
\end{enumerate}
\end{corollary}
\begin{proof}
The functor $f$ is homotopic to a composition of the functors $u\colon \cat{C}\rt \cat{C}[W^{-1}]$ and $g\colon \cat{C}[W^{-1}]\rt \cat{D}$. Since homotopic maps induce equivalent pullbacks of left fibrations, we may assume that $f=gu$. It that case, the square \eqref{diag:mappingspaces} decomposes into two such squares, involving $u(c)/\cat{C}[W^{-1}]$. The functor $g$ is fully faithful if and only if for each $c'\in \cat{C}[W^{-1}]$, the map $c'/\cat{C}[W^{-1}]\rt g^*(g(c')/\cat{D})$ is a fiberwise equivalence of left fibrations over $\cat{C}[W^{-1}]$. Since the functor $u$ is a bijection on vertices (Remark \ref{rem:localizationaspushout}), we see that $g$ is fully faithful if and only if the second map in the composite
$$\xymatrix{
c/\cat{C}\ar[r] & u^*(u(c)/\cat{C}[W^{-1}]) \ar[r] & f^*(f(u)/\cat{D})
}$$
is a fiberwise equivalence of $W$-local left fibrations over $\cat{C}$. But this is equivalent to the composite map being a $W$-local covariant weak equivalence. Indeed, using that the left fibration $u(c)/\cat{C}[W^{-1}]\rt \cat{C}[W^{-1}]$ is covariantly equivalent to the inclusion $\{u(c)\}\rt \cat{C}[W^{-1}]$, one sees that the first map in the above composite is a model for the derived unit map of the adjunction \eqref{eq:equivalencebetweenrelativefunctors} and therefore a $W$-local weak equivalence.
\end{proof}
\begin{corollary}\label{cor:mappingspacesfromfibrantreplacement}
Let $(\cat{C}, W)$ be a relative category and let $c, d\in \cat{C}$ be two objects in $\cat{C}$. Let $p\colon X\rt \cat{C}$ be a fibrant replacement of the map $\{c\}\rt \cat{C}$ in the $W$-local covariant model structure. Then the fiber $X_d=p^{-1}(d)$ is a model for the mapping space $\Map_{\cat{C}[W^{-1}]}(c, d)$.
\end{corollary}
\begin{proof}
This follows immediately from the fact that $\{c\}\rt c/\cat{C}\rt u^*(u(c)/\cat{C}[W^{-1}])$ is a fibrant replacement of $\{c\}\rt \cat{C}$ in the $W$-local model structure, together with the fact that the fibers of $u^*(u(c)/\cat{C}[W^{-1}])\rt \cat{C}$ over an object $d$ are a model for $\Map_{\cat{C}[W^{-1}]}(c, d)$.
\end{proof}
\begin{remark}
Let $(\cat{C}, W)$ be a relative $\infty$-category and let $\mb{D}$ be the full simplicial subcategory of $\sSet/\cat{C}$ on the $W$-local left fibrations $X\rt \cat{C}$ over $\cat{C}$ that are weakly representable, i.e.\ for which there exists an object $x\in \cat{C}$ and a $W$-local equivalence $\{x\}\rt X$ over $\cat{C}$. The fibrant simplicial category $\mb{D}$ is Dwyer-Kan equivalent to the simplicial category of left fibrations over $\cat{C}[W^{-1}]$ that are weakly representable. Consequently, $\mb{D}$ provides a model for $\cat{C}[W^{-1}]$ by the Yoneda lemma \cite[Proposition 5.1.3.1]{lur09}. 
\end{remark}

\section{Localization categories with hypercovers}\label{sec:localizinghypercovers}
In this section we will apply the results of the previous section in the case where $(\cat{C}, W)$ is a relative $\infty$-category in which the set of maps $W$ is stable under base change. In this case, we will give a description of the mapping spaces in $\cat{C}[W^{-1}]$ in terms of `cocycle categories', in the terminology of \cite{jar06} (see Corollary \ref{cor:mappingspaceoffractioncat}). 
\begin{definition}\label{def:categorywithhypercovers}
A relative $\infty$-category $(\cat{C}, W)$ is called an \emph{$\infty$-category with hypercovers} if every solid diagram
$$\xymatrix{
D'\ar@{..>}[r]^{w'}\ar@{..>}[d] & D\ar[d]\\
C'\ar[r]_{w\in W} & C
}$$
in which the bottom map $w$ is contained in $W$, admits a pullback as indicated such that the map $w'$ is contained in $W$. In this case, we will usually call the maps in $W$ hypercovers and denote them by $\rto{\sim}$.
\end{definition}
\begin{example}
Any category of fibrant objects in the sense of Brown \cite{bro73} is a category with hypercovers, where the hypercovers are the acyclic fibrations. Note that the acyclic fibrations generate the weak equivalences under the 2-out-of-3 property (so that localizing at both sets of maps yields the same result), but that the weak equivalences need not be stable under pullback.
\end{example}
\begin{remark}
Any $\infty$-category can be obtained by localizing an (ordinary) category with hypercovers. Indeed, for any $\infty$-category $\cat{C}$, consider the full subcategory $\mb{D}\subseteq \cat{sSet}/\cat{C}$ of left fibrations $X\rt \cat{C}$ that are weakly representable, i.e.\ for which there exists an object $x\in \cat{C}$ together with a covariant weak equivalence $\{x\}\rt X$ over $\cat{C}$. Together with the trivial fibrations between such left fibrations in $\mb{D}$, the category $\mb{D}$ forms a category with hypercovers, whose localization $\mb{D}[W^{-1}]$ is equivalent to $\cat{C}$ by the Yoneda lemma \cite[Proposition 5.1.3.1]{lur09}.
\end{remark}
\begin{example}\label{ex:sspace}
Let $\cat{C}=\Fun(N(\mb{\Delta})^\mm{op}, \sS)$ be the $\infty$-category of simplicial spaces. This is an $\infty$-category with hypercovers, with hypercovers given by the trivial Kan fibrations, i.e.\ those maps $Y\rt X$ of simplicial spaces for which the relative (homotopy) matching maps $Y_n\rt \mm{M}_nY\times_{\mm{M}_nX} X_n$ induce surjections on path components. If one presents $\cat{C}$ by the Reedy model structure on the category $\Fun(\mb{\Delta}^\mm{op}, \sSet)$ of bisimplicial sets, then a Reedy fibration $p\colon Y\rt X$ presents a hypercover in $\Fun(N(\mb{\Delta})^\mm{op}, \sS)$ if and only if its image under the functor
$$\xymatrix{
\Fun(\mb{\Delta}^\mm{op}, \mm{ev}_n)\colon \Fun(\mb{\Delta}^\mm{op}, \sSet)\ar[r] & \Fun(\mb{\Delta}^\mm{op}, \Set)
}$$
is a trivial fibration of simplicial sets for each $n\geq 0$. 

The associated localization $\Fun(N(\mb{\Delta})^\mm{op}, \sS)[W^{-1}]$ is quite unwieldy, but one gets a simple result if one restricts attention to the full subcategory $\cat{Kan}\subseteq \Fun(N(\mb{\Delta})^\mm{op}, \sS)$ on those simplicial spaces $X$ that satisfy the \emph{Kan condition}, i.e.\ whose (homotopy) matching maps $X_n\rt \mm{M}_{\Lambda^i[n]}X$ induce surjections on path components for each horn inclusion. The natural map from such a simplicial space $X$ to the constant diagram on its (homotopy) colimit turns out to be a hypercover (as we will discuss in \cite{nui16b}). From this it follows easily that the constant simplicial diagrams form a right deformation retract of the relative $\infty$-category $\cat{Kan}$, so that $\cat{Kan}[W^{-1}]\simeq \sS$.
\end{example}
\begin{example}\label{ex:algebras}
Let $\cat{O}$ be a $1$-coloured $\infty$-operad and let $\cat{C}=\Fun(N(\mb{\Delta})^\mm{op}, \cat{Alg}_{\cat{O}}(\sS))$ be the $\infty$-category of simplicial diagrams in the $\infty$-category of $\cat{O}$-algebras in spaces. If one describes $\cat{O}$ by a fibrant dendroidal set, then one may model $\cat{C}$ by the Reedy model structure on $\mm{Fun}(\mb{\Delta}^\mm{op}, \cat{dSet}^\mm{cov}/\cat{O})$ with respect to the covariant model structure on $\cat{dSet}/\cat{O}$ (see Variant \ref{var:operadic}). The $\infty$-category $\cat{C}$ is an $\infty$-category with hypercovers, with the hypercovers given by those maps of simplicial $\cat{O}$-algebras whose underlying map of simplicial spaces is a trivial Kan fibration. The resulting $\infty$-category with hypercovers can be used to study simplicial resolutions of $\cat{O}$-algebras by \emph{free} $\cat{O}$-algebras.

As in Example \ref{ex:sspace}, to get a well-behaved localization one should restrict attention to the full subcategory $\cat{Kan}_{\cat{O}}\subseteq \Fun(N(\mb{\Delta})^\mm{op}, \cat{Alg}_{\cat{O}}(\sS))$ on the simplicial $\cat{O}$-algebras whose underlying simplicial space is a Kan complex. In that case, the localization $\cat{Kan}_{\cat{O}}[W^{-1}]\simeq \cat{Alg}_{\cat{O}}$ is equivalent to the original $\infty$-category of $\cat{O}$-algebras. Indeed, since the forgetful functor $\cat{Alg}_{\cat{O}}(\sS)\rt \sS$ preserves colimits of simplicial diagrams, it follows from Example \ref{ex:sspace} that the map from a simplicial diagram in $\cat{Kan}_{\cat{O}}$ to the constant diagram on its colimit is a hypercover.
\end{example}
\begin{example}
One can also consider more geometric variations of Example \ref{ex:sspace}: for example, one can take $\cat{C}=\Fun(\mb{\Delta}^\mm{op}, \cat{Mfd})$ to be the category of simplicial manifolds, with the hypercovers given by those maps of simplicial manifolds $Y\rt X$ whose relative matching maps $Y_n\rt \mm{M}_nY\times_{\mm{M}_nX} X_n$ are surjective submersions. In the setting of derived algebraic geometry, one can take $\cat{C}$ to be the $\infty$-category of simplicial derived schemes, with hypercovers given by those maps whose relative matching maps are smooth (or \'etale) surjections. Again, the localizations of these categories at their hypercovers can be somewhat complicated, but become easier if one considers only those  simplicial objects satisfying a (truncated) version of the Kan condition (i.e.\ the subcategories of Lie $n$-groupoids or derived Artin $n$-groupoids). We will come back to this in \cite{nui16b}.
\end{example}

Let $(\cat{C}, W)$ be an $\infty$-category with hypercovers. The space of maps $\Map_{\cat{C}[W^{-1}]}(c, d)$ admits the following description, which resembles the description in the 1-categorical setting from \cite{cis10b, jar06}: the space $\Map_{\cat{C}[W^{-1}]}(c, d)$ can be obtained as the group completion (i.e.\ the Kan fibrant replacement) of the $\infty$-category of spans
$$\xymatrix{
c & \tilde{c}\ar[r]\ar[l]_\sim & d
}$$
where $\tilde{c}\rt c$ is contained in $W$, where a morphism is (roughly) a commuting diagram of the form
$$\xymatrix@R=0.6pc{
 & \tilde{c}\ar[ld]_{\sim}\ar[dd]\ar[rd] & \\
c & & d.\\
& \tilde{c}'\ar[lu]^{\sim}\ar[ru]
}$$
One expects this space to depend functorially on the object $d\in \cat{C}$, simply by postcomposing spans with maps $d\rt d'$ in $\cat{C}$. To make this more precise, let $c$ an object in the $\infty$-category $\cat{C}$ and let
$$
\cat{C}_W^{/c}\subseteq \cat{C}^{/c}:=\mm{Fun}(\Delta[1], \cat{C})\times_{\mm{Fun}(\{1\}, \cat{C})} \{c\}
$$
denote the full sub-$\infty$-category of the (alternative) slice category $\cat{C}^{/c}$ on those arrows $\tilde{c}\rt c$ that are contained in $W$. 
\begin{definition}\label{def:spanfibration}
Let $(\cat{C}, W)$ be a relative $\infty$-category and let $c\in \cat{C}$. Define $\pi\colon H(c)\rt \cat{C}$ to be the map of quasicategories
$$\xymatrix{
\pi\colon H(c):= \cat{C}_W^{/c}\times_{\Fun(\dau_2\Delta[2], \cat{C})} \Fun(\Lambda^0[2], \cat{C})\ar[r]^-{\mm{ev}_2} & \cat{C} 
}$$
which sends a span $c\lt \tilde{c}\rt d$ with $\tilde{c}\rt c$ in $W$ to its endpoint $d$.
\end{definition}
\begin{remark}\label{rem:relativespans}
For any two objects $c, d\in \cat{C}$, let us denote by $\mm{Span}^W_{\cat{C}}(c, d)$ the fiber of the map $\pi\colon H(c)\rt \cat{C}$ over $d$. More concretely, the quasicategory $\mm{Span}^W_{\cat{C}}(c, d)$ is the full sub-$\infty$-category
$$
\mm{Span}^W_{\cat{C}}(c, d)\subseteq \mm{Fun}(\Lambda^0[2], \cat{C})\times_{\mm{Fun}(\{1, 2\}, \cat{C})} \{(c, d)\}
$$
consisting of those spans $c\lt \tilde{c}\rt d$ for which the left map is contained in $W$. If $W$ contains all maps in $\cat{C}$ (resp.~ only the equivalences), we will denote this $\infty$-category simply by $\mm{Span}_{\cat{C}}(c, d)$ (resp.~ $\mm{Span}^\mm{eq}_{\cat{C}}(c, d)$).
\end{remark}
Although the fibers of the map $\pi\colon H(c)\rt \cat{C}$ are quasicategories, rather than Kan complexes, we will argue that $\pi$ is nonetheless quite close to a fibrant replacement of the map $\{c\}\rt \cat{C}$ in the $W$-local covariant model structure on $\sSet/\cat{C}$. As a start, we show that $\pi$ is indeed weakly equivalent to the inclusion $\{c\}\rt \cat{C}$ in the $W$-local covariant model structure.
\begin{lemma}\label{lem:leftkanbycodomainfib}
Let $f\colon \cat{D}\rt \cat{C}$ be a map of quasicategories and consider the diagram
\begin{equation}\label{diag:degeneracycovwe}\vcenter{\xymatrix{
\cat{D}\ar[rd]_f\ar[rr]^-{(\mm{id}, s_0)}_>{\;}="s" & & P_f:=\mm{Fun}(\{0\}, \cat{D})\times_{\Fun(\{0\}, \cat{C})} \Fun(\Delta[1], \cat{C})\\
& \cat{C}\ar"s";[]^\pi
}}\end{equation}
where $\pi$ is given by evaluation at $1$. Then $\pi$ is a cocartesian fibration and $(\mm{id}, s_0)$ is a covariant weak equivalence over $\cat{C}$.
\end{lemma}
\begin{proof}
The map $\pi$ is a cocartesian fibration by (the opposite of) \cite[Corollary 2.4.7.12]{lur09}. The diagram \eqref{diag:degeneracycovwe} can be extended to a diagram
$$\xymatrix{
\cat{D}\ar[r]^-{(\mm{id}, s_0)}\ar[d] & P_f \ar[d]\ar[r]^r & \cat{D}\ar[d]\\
\cat{C}\ar[rd] \ar[r]^-{s_0} & \mm{Fun}(\Delta[1], \cat{C})\ar[r]_-{\mm{ev}_0}\ar[d]^{\mm{ev}_1} & \cat{C}\\
& \cat{C}.
}$$
where both squares are cartesian. In particular, it follows that $(\mm{id}, s_0)$ has a retract $r$, given by the base change of $\mm{ev}_0\colon \mm{Fun}(\Delta[1], \cat{C})\rt \cat{C}$. By \cite[Lemma 2.7]{heu13}, to prove that the map $(\mm{id}, s_0)$ is a covariant trivial cofibration, it suffices to construct a homotopy $\Delta[1]\times H(c)\rt H(c)$ from the composition $(\mm{id}, s_0\mm{ev}_0)$ to the identity map, relative to $\cat{D}$. This homotopy is simply the base change of the homotopy $\Delta[1]\times\cat{C}^{\Delta[1]}\rt \cat{C}^{\Delta[1]}$ from $s_0\mm{ev}_0$ to the identity that is adjoint to the map
$$\xymatrix{
\Delta[1]\times\Delta[1]\times\mm{Fun}(\Delta[1], \cat{C}) \ar[rrr]^-{\mm{min}\times \mm{Fun}(\Delta[1], \cat{C})} & & & \Delta[1]\times \mm{Fun}(\Delta[1], \cat{C})\ar[r]^-{\mm{ev}} & \cat{C}
}$$
where $\mm{min}\colon \Delta[1]\times \Delta[1]\rt \Delta[1]$ is the functor taking the minimum. 
\end{proof}
Since the map $\pi\colon H(c)\rt \cat{C}$ of Definition \ref{def:spanfibration} is precisely the map $\pi\colon P_f\rt \cat{C}$ associated to the functor $f\colon \cat{C}^{/c}_W\rt \cat{C}$ by the above construction, it follows that $\pi$ is a cocartesian fibration.
\begin{lemma}\label{lem:Wlocalequivalence}
Let $(\cat{C}, W)$ be a relative category. The canonical map $\{c\}\rt H(c)$ over $\cat{C}$ corresponding to the constant span $c \lt c \rt c$ is an equivalence in the $W$-local covariant model structure.
\end{lemma}
\begin{proof}
The map $\{c\}\rt H(c)$ factors as
$$\xymatrix{
\{c\}\ar[r] & \cat{C}_W^{/c}\ar[r]^g & H(c)
}$$
where the second map $g$ is the covariant weak equivalence of Lemma \ref{lem:leftkanbycodomainfib}. It therefore suffices to show that the inclusion $\{c\}\rt \cat{C}_W^{/c}$ of the identity map on $c$ is a trivial cofibration in the $W$-local model structure on $\sSet/\cat{C}$.

To this end, let $\ol{\cat{W}}\subseteq \cat{C}$ be the the maximal sub-simplicial set of $\cat{C}$ whose arrows are given by the closure $\ol{W}$ of $W$ under the 2-out-of-3 property. Then $\ol{\cat{W}}$ is itself a quasicategory and the inclusion $\ol{\cat{W}}_W^{/c}\subseteq \cat{C}_W^{/c}$ is an isomorphism: indeed, if $\alpha\colon \Delta[n+1]\rt \cat{C}$ has edges $\alpha(i)\rt \alpha(n+1)$ in $W$ for all $i$, then all edges $\alpha(i)\rt \alpha(j)$ are contained in $\ol{W}$ by the 2-out-of-3 property. Since the left Quillen functor $\sSet^\mm{cov}/(\ol{\cat{W}}, W)\rt \sSet^\mm{cov}/(\cat{C}, W)$ preserves all $W$-local equivalences, it will suffice to show that the inclusion $\{c\} \rt  \ol{\cat{W}}_W^{/c}$ is a $W$-local equivalence in $\sSet/\ol{\cat{W}}$.

But the $W$-local model structure on $\sSet/\ol{\cat{W}}$ agrees with the Kan-Quillen model structure by Lemma \ref{lem:characterizingWlocalfibrations}. The result now follows from the fact that there is a concrete homotopy
$$\xymatrix{
\Delta[1]\times \ol{\cat{W}}_W^{/c} \ar[r] & \ol{\cat{W}}_W^{/c}
}$$
from the identity map to the constant map whose value is the identity on $c$. Indeed, this is simply given by the map adjoint to
$$\xymatrix{
\Delta[1]\times\Delta[1]\times \mm{Fun}(\Delta[1], \ol{\cat{W}}) \ar[rrr]^-{\mm{max}\times \mm{Fun}(\Delta[1], \ol{\cat{W}})} & & & \Delta[1]\times \mm{Fun}(\Delta[1], \ol{\cat{W}})\ar[r]^-{\mm{ev}} & \ol{\cat{W}}.
}$$
restricted to the simplicial subset $\ol{\cat{W}}_W^{/c}\subseteq \mm{Fun}(\Delta[1], \ol{\cat{W}})$.
\end{proof}
Essentially the only obstruction to $\pi\colon H(c)\rt \cat{C}$ being a fibrant replacement of $\{c\}\rt \cat{C}$ in $\sSet^\mm{cov}/(\cat{C}, W)$ is the fact that it is not a left fibration. More precisely, we will now show that there is a covariant weak equivalence from $\pi$ to a $W$-local left fibration $|\pi|\colon |H(c)|\rt \cat{C}$ whose fibers are simply the group-completions of the fibers of $\pi$. 
\begin{lemma}\label{lem:completionofcocartesianfibrations}
Consider a map $f$ of cocartesian fibrations over a simplicial set $S$
$$\xymatrix{
X\ar[rr]^f\ar[rd]_p & & Y\ar[ld]^q\\
 & S. &
}$$
Then $f$ is a covariant weak equivalence if and only if for each $s\in S$, the map of fibers $X_s\rt Y_s$ is an equivalence of simplicial sets in the Kan-Quillen model structure. Consequently, for any map $\alpha\colon S'\rt S$, the functor $\alpha^*\colon \sSet/S\rt \sSet/S'$ preserves covariant weak equivalences between cocartesian fibrations.
\end{lemma}
\begin{proof}
This is a special case of \cite[4.1.2.17]{lur09}. Alternatively, one can easily reduce to the case where $q\colon Y\rt S$ is a left fibration, so that $f$ preserves cocartesian edges. Now apply the (marked) straightening functor to the cocartesian fibrations $p$ and $q$ and let $\mm{St}(f)\colon \mm{St}(p^\natural)\rt \mm{St}(q^\natural)$ be the map of the underlying $\mf{C}[S]$-indexed diagrams of (unmarked) simplicial sets. The left fibrations associated to $p$ and $q$ can then be described as the (unmarked) unstraightening of Kan fibrant replacements of $\mm{St}(p^\natural)$ and $\mm{St}(q^\natural)$. In particular, these associated left fibrations are (fiberwise) equivalent if and only if $\mm{St}(f)$ is a pointwise Kan-Quillen equivalence of $\mf{C}[S]$-diagrams of simplicial sets. But the value of $\mm{St}(f)$ at a point $c\in \mf{C}[S]$ is categorically equivalent to the map of fibers $X_c\rt Y_c$.
\end{proof}
In particular, the natural map from $\pi\colon H(c)\rt \cat{C}$ to the associated left fibration $|\pi|\colon |H(c)|\rt \cat{C}$ realizes the fibers of $|\pi|$ as Kan fibrant replacements of the fibers of $\pi$. To see that the left fibration $|\pi|$ is also fibrant in the $W$-local model structure, we need the following two lemmas:
\begin{lemma}\label{lem:adjunctionyieldsequivalence}
Let $p\colon X\rt \Delta[1]$ be a cocartesian fibration and let $|p|\colon |X|\rt \Delta[1]$ be the associated left fibration. If $p$ is also a cartesian fibration, then $|p|$ is a Kan fibration.
\end{lemma}
\begin{proof}
Since $p$ is a cocartesian fibration there exists a map
$$\xymatrix@R=1.4pc{
X_0\times \Delta[1]\ar[rd]_{\pi_2}\ar[rr]^f & & X\ar[ld]^p\\
& \Delta[1] &
}$$
such that $f\big|X_0\times \{0\} = \mm{id}$ and $f\big|\{x\}\times\Delta[1]$ is a $p$-cocartesian edge \cite[Proposition 5.2.1.4]{lur09}. Since the left fibration $|\pi_2|$ associated to the cocartesian fibration $\pi_2$ is clearly a Kan fibration, it suffices to show that the map $f\colon X_0\times\{1\}\rt X_1$ is a homotopy equivalence (see Remark \ref{rem:kanfibrationover1}). This follows from \cite[5.2.2.8]{lur09}, which asserts that there exists a functor $g\colon X_1\rt X_0$ (since $p$ is a cartesian fibration) and natural tranformations $\mm{id}_{X_0}\rt gf$ and $fg\rt \mm{id}_{X_1}$.
\end{proof}
\begin{lemma}\label{lem:pullingbackspans}
Let $\cat{C}$ be a quasicategory and $c\in \cat{C}$. Denote by $p\colon \mm{Span}(c)\rt \cat{C}$ the cocartesian fibration 
$$\xymatrix{
\mm{Fun}(\Lambda^0[2], \cat{C})\times_{\mm{Fun}(\{1\}, \cat{C})} \{c\} \ar[r] & \mm{Fun}(\{2\}, \cat{C}) = \cat{C}
}$$
sending $c\lt e\rt d$ to $d$. Then the following results hold:
\begin{enumerate}
\item For any span $x=\big[c\lt e\rt d\big]\in\mm{Span}(c)$, the obvious map
$$\xymatrix{
\mm{Span}(c)/x \ar[r] & \mm{Fun}(\Delta[1], \cat{C})/(e\rt d)
}$$
is a trivial fibration.

\item Let $\alpha\colon d'\rt d$ be an arrow in $\cat{C}$ with the property that all pullbacks along $\alpha$ exist (i.e. for any arrow $\beta\colon e\rt d$, the fiber product of $\alpha$ and $\beta$ exists in $\cat{C}$). Then every object $x\in p^{-1}(d)$ admits a $p$-cartesian lift $\tilde{\alpha}\colon x'\rt x$ of $\alpha$.
\end{enumerate}
\end{lemma}
\begin{proof}
For (1), observe that a lifting problem of the form
$$\xymatrix{
\dau\Delta[n]\ar[d]\ar[r] & \mm{Span}(c)/x\ar[d]\\
\Delta[n]\ar[r]\ar@{..>}[ru] & \mm{Fun}(\Delta[1], \cat{C})/(e\rt d)
}$$
corresponds to a lifting problem of the form
\begin{equation}\label{diag:spantoarrowisrightfib}\vcenter{\xymatrix{
\Lambda^n[n+1]\ar[d]\ar[r] & \mm{Span}(c)\ar[d]\ar[r] & \cat{C}^{/c}\ar[d]\\
\Delta[n+1]\ar[r]\ar@{..>}[ru] & \mm{Fun}(\Delta[1], \cat{C})\ar[r]_-{\mm{dom}} & \cat{C}. 
}}\end{equation}
Since $\cat{C}^{/c}\rt \cat{C}$ is a right fibration and the right square is a pullback, such a lift exists.

For (2), observe that $p$ is the composition of the right fibration $q\colon \mm{Span}(c)\rt \Fun(\Delta[1], \cat{C})$ appearing in \eqref{diag:spantoarrowisrightfib} and the map $\mm{ev}_1\colon \Fun(\Delta[1], \cat{C})\rt \cat{C}$. Given $\alpha\colon d'\rt d$ in $\cat{C}$ and $x=\big[c\lt e\rt d\big]$ in $\mm{Span}(c)$, let $\ol{\alpha}$ and $\tilde{\alpha}$ be the arrows in $\Fun(\Delta[1], \cat{C})$, resp.\ $\mm{Span}(c)$ classifying the diagrams
\begin{equation}\label{diag:cartesianlift}
\ol{\alpha}\colon\hspace{4pt}\vcenter{\xymatrix{
e'\ar[r]\ar[d]_{\alpha'} & d'\ar[d]^\alpha\\
e\ar[r]_\beta & d.
}}\qquad\qquad \tilde{\alpha}\colon\hspace{4pt}\vcenter{\xymatrix{
c\ar@{=}[d] & e'\ar[l]\ar[r]\ar[d]_{\alpha'} & d'\ar[d]^\alpha\\
c & e\ar[l]\ar[r]_\beta & d.
}} 
\end{equation}
where $\ol{\alpha}$ is a cartesian square (which exists by assumption) and $\tilde{\alpha}$ is obtained from $\ol{\alpha}$ by composing $\alpha'$ and $e\rt c$. Clearly $\tilde{\alpha}$ is a lift of $\alpha$ against $p$ with endpoint $x$.

The arrow $\tilde{\alpha}$ is a $q$-cartesian lift of $\ol{\alpha}$ since $q$ is a right fibration \cite[Proposition 2.4.2.4]{lur09}. It suffices to check that $\ol{\alpha}$ is a cartesian lift of $\alpha$ against $\mm{ev}_1$. To see this, we have to check that each diagram
$$\xymatrix{
\{n\}\times\Delta[1]\ar[d]\ar[rd]^{\tilde{\alpha}} \\
\Delta[n]\times\{1\}\cup \dau\Delta[n]\times\Delta[1]\ar[d]\ar[r] & \Fun(\Delta[1], \cat{C})\ar[d]^{\mm{ev}_1}\\
\Delta[n]\times\Delta[1]\ar[r]\ar@{..>}[ru] & \cat{C}
}$$
with $n\geq 1$ admits a diagonal lift \cite[Proposition 2.4.1.8]{lur09}. Unwinding the definitions, this lifting problem is equivalent to a lifting problem
$$\xymatrix{
\dau\Delta[n]\ar[d]\ar[r]^-f & \Fun(\Delta[1]\times \Delta[1], \cat{C})\ar[d]^\psi\\
\Delta[n]\ar[r]\ar@{..>}[ru] & \Fun(\Lambda^2[2], \cat{C})
}$$
where $f(n)\in \Fun(\Delta[1]\times \Delta[1], \cat{C})$ is given by $\tilde{\alpha}$, which we defined to be a right Kan extension of its restriction to $\Lambda^2[2]$. The desired lift now exists by by \cite[Lemma 4.3.2.12]{lur09}.
\end{proof}
Combining the previous results, we obtain the following description of the mapping spaces of $\cat{C}[W^{-1}]$:
\begin{corollary}\label{cor:mappingspaceoffractioncat}
Let $(\cat{C}, W)$ be an $\infty$-category with hypercovers and let $\pi\colon H(c)\rt \cat{C}$ be the cocartesian fibration from Definition \ref{def:spanfibration}. Then
\begin{enumerate}
\item the associated left fibration $|H(c)|\rt \cat{C}$ is a $W$-local left fibration (i.e. a Kan fibration when restricted to the arrows of $W$).
\item the map $H(c)\rt |H(c)|$ realizes the fiber of the latter over $d\in \cat{C}$ as the group completion of the $\infty$-category $\mm{Span}^W_{\cat{C}}(c, d)$ of Remark \ref{rem:relativespans}.
\end{enumerate}
In particular, $\Map_{\cat{C}[W^{-1}]}(c, d)$ is equivalent to the group-completion of $\mm{Span}^W_{\cat{C}}(c, d)$.
\end{corollary}
\begin{proof}
By Lemma \ref{lem:adjunctionyieldsequivalence}, assertion (1) follows once we show that for any hypercover $w\colon d'\rt d$ in $\cat{C}$ and any span $x$ in $\mm{Span}^W_{\cat{C}}(c, d)$, there is a locally $\pi$-cartesian edge $\tilde{w}\colon x'\rt x$ in $H(c)$ lifting $w$ (i.e. a cartesian edge with respect to the restriction $w^*H(c)\rt \Delta[1]$ of $\pi$). To see this, note that there is a fully faithful inclusion $H(c)\rt \mm{Span}(c)$ of cocartesian fibrations over $\cat{C}$. The cartesian lift $\tilde{w}\colon \Delta[1]\rt \mm{Span}(c)$ of $w$ provided by Lemma \ref{lem:pullingbackspans} (classifying the right diagram in \eqref{diag:cartesianlift}) is in fact an arrow in $H(c)$, since hypercovers are stable under base change and composition. Since $\tilde{w}$ is a cartesian edge in $\mm{Span}(c)$, it remains so in the full subcategory $H(c)$. Part (2) follows from part (1) and Lemma \ref{lem:completionofcocartesianfibrations}.
\end{proof}
\begin{corollary}\label{cor:fullyfaithfulonhypercoverlocalization}
Let $(\cat{C}, W)$ be an $\infty$-category with hypercovers and let $f\colon \cat{C}\rt \cat{D}$ be a functor sending the maps in $W$ to equivalences in $\cat{D}$. For each $c, d\in \cat{C}$, consider the induced functor
$$\xymatrix{
\cat{Span}^W_{\cat{C}}(c, d) \ar[r]^-f & \cat{Span}_{\cat{D}}^\mm{eq}(f(c), f(d))
}$$
from the $\infty$-category of spans $c\lt \tilde{c}\rt d$ with the left map in $W$ to the $\infty$-category of spans $f(c)\lt x\rt d$ with the left map an equivalence in $\cat{D}$. This map of quasicategories is a Kan-Quillen equivalence of simplicial sets for all $c$ and $d$ if and only if the functor $\cat{C}[W^{-1}]\rt \cat{D}$ is fully faithful.
\end{corollary}
\begin{proof}
For any object $c\in \cat{C}$, consider the diagram
$$\xymatrix{
c/\cat{C}\ar[r] \ar[d] & H(c)\ar[d]\ar[r]  & \cat{C}\ar[d]^f\\
f(c)/\cat{D} \ar[r] & H^\mm{eq}(f(c))\ar[r] & \cat{D}
}$$
where $H^\mm{eq}(f(c))$ arises as in Definition \ref{def:spanfibration} from the relative $\infty$-category $(\cat{D}, \mm{eq})$. In light of  Corollary \ref{cor:fullyfaithfuloutoflocalization}, the functor $\cat{C}[W^{-1}]\rt \cat{D}$ is fully faithful if and only if for any $c\in \cat{C}$, the left vertical map induces a $W$-local weak equivalence $c/\cat{C}\rt f^*(f(c)/\cat{D})$ of left fibrations over $\cat{C}$. Since the map $f(c)/\cat{D}\rt H^\mm{eq}(f(c))$ is a covariant weak equivalence between left fibrations over $\cat{D}$, this is equivalent to the composition
$$\xymatrix{
c/\cat{C}\ar[r] & H(c) \ar[r] & f^*H^\mm{eq}(f(c))
}$$
being a $W$-local weak equivalence. Since the first map is a $W$-local weak equivalence, this is in turn equivalent that the second map is a $W$-local weak equivalence. But $H(c)\rt \cat{C}$ and $f^*H^\mm{eq}(f(c))\rt \cat{C}$ are cocartesian fibrations whose restriction to each arrow in $W$ is cartesian. The second map is therefore a $W$-local weak equivalence if and only if the induced maps on fibers are Kan-Quillen equivalences, which is precisely the assertion of the lemma.
\end{proof}
Although the above gives a rather concrete description of the mapping spaces in $\cat{C}[W^{-1}]$, it is not completely satisfactory as it lacks a description of the composition in $\cat{C}[W^{-1}]$. We will give a more detailed description of $\cat{C}[W^{-1}]$ in terms of spans in the next section.

\section{Bicategories of fractions}\label{sec:2catoffractions}
The aim of this section is to extend the description of the mapping spaces of $\cat{C}[W^{-1}]$ in terms of spans to a description of the entire category $\cat{C}[W^{-1}]$ (including the composition of maps). Since pullbacks along hypercovers exist and remain hypercovers, there is an obvious notion of composition of spans, given by forming the pullback
$$\xymatrix@rd@C=1pc@R=1pc{
\tilde{c}\times_d \tilde{d} \ar[r] \ar[d]_-{\scriptscriptstyle\sim} & \tilde{d} \ar[d]^-{\scriptscriptstyle\sim}\ar[r] & e\\
\tilde{c}\ar[r]\ar[d]_-{\scriptscriptstyle\sim} & d\\
c
}$$
and taking the resulting total span. Unfortunately, this composition is only determined up to equivalence and is a priori determined at the level of \emph{categories} of spans, rather than their group completions. To address these issues, it will be useful to work in a setting of $(\infty, 2)$-categories which allows for a description of an $(\infty, 2)$-category $\mm{Span}(\cat{C}, W)$ whose mapping categories are the above categories of spans, so that the associated freely generated $\infty$-category $|\mm{Span}(\cat{C}, W)|$ (obtained by group-completing all mapping categories) is a model for the localization $\cat{C}[W^{-1}]$. We will first recall the piece of $(\infty, 2)$-category theory we need and provide a model for $\mm{Span}(\cat{C}, W)$ in Section \ref{sec:2catofspans} (see Construction \ref{con:spancategory} and Theorem \ref{thm:bicatofspans}).

\subsection{Some 2-category theory}\label{sec:basicson2cats}
To stay relatively close to quasicategories, we will use the following model for $(\infty, 2)$-categories:
\begin{definition}\label{def:segalobject}
Let $X\colon \mb{\Delta}^\mm{op}\rt \sSet$ be a bisimplicial set. We will say that $X$ is a \emph{Segal object} if 
\begin{enumerate}
\item $X$ is Reedy fibrant, where $\sSet$ carries the \emph{Joyal model structure}.
\item $X_0$ is a Kan complex.
\item the Segal maps $X_n\rt X_1\times_{X_0} \cdots \times_{X_0} X_n$ are equivalences of quasicategories.
\end{enumerate}
If $X\colon \mb{\Delta}^\mm{op}\rt \sSet$ is a Segal object, let $K(X)\colon \mb{\Delta}^\mm{op}\rt \sSet$ be the simplicial object whose $n$-th object is the maximal Kan complex $K(X_n)$ contained in $X_n$. Since $K$ preserves fibrations, it is easy to see that $K(X)$ is a Segal space in the usual sense of \cite{rez98b}. We will say that the Segal object $X$ is \emph{complete} if $K(X)$ is a complete Segal space in the usual sense.
\end{definition}
\begin{lemma}
The category $\mm{Fun}(\mb{\Delta}^\mm{op}, \sSet)$ carries a model structure (which we call the \emph{2-categorical model structure}) whose cofibrations are the monomorphisms and whose fibrant objects are the complete Segal objects. Furthermore, this model category is Quillen equivalent to the model category for 2-fold complete Segal spaces.
\end{lemma}
\begin{proof}
Recall from \cite{joy06} that there is a Quillen equivalence 
$$
i_0\colon \sSet_\mm{Joy} \leftrightarrows \mm{Fun}(\mb{\Delta}^\mm{op}, \sSet)_\mm{CSS} \colon \mm{ev}_0
$$
between the Joyal model structure and the (injective) model structure for complete Segal spaces. This Quillen equivalence induces a Quillen equivalence 
$$\xymatrix{
(i_0)_*\colon \mm{Fun}(\mb{\Delta}^\mm{op}, \sSet_\mm{Joy}) \ar@<1ex>[r] &  \mm{Fun}(\mb{\Delta}^\mm{op}, \mm{Fun}(\mb{\Delta}^\mm{op}, \sSet)_\mm{CSS})\colon (\mm{ev}_0)_*\ar@<1ex>[l]
}$$
between the between the (injective, equivalently Reedy) model categories of simplicial diagrams in $\sSet$ with the Joyal model structure, resp.~ in the model structure of complete Segal spaces. 

The right hand model structure admits a Bousfield localization whose fibrant objects are the 2-fold complete Segal spaces and the left model structure admits a Quillen equivalent Bousfield localization. The fibrant objects in this Bousfield localization are those Reedy fibrant diagrams $X\colon \mb{\Delta}^\mm{op}\rt \sSet_\mm{Joyal}$ that are levelwise equivalent to a diagram of the form $(\mm{ev}_0)_*Y$, where $Y\colon \mb{\Delta}^\mm{op}\rt \mm{Fun}(\mb{\Delta}^\mm{op}, \sSet)$ is a 2-fold complete Segal space. The result now follows from the observation that a Reedy fibrant diagram $Y\colon \mb{\Delta}^\mm{op}\rt \mm{Fun}(\mb{\Delta}^\mm{op}, \sSet)_\mm{CSS}$ is a 2-fold complete Segal space if and only if $(\mm{ev}_0)_*Y$ is a complete Segal object in the sense of Definition \ref{def:segalobject}: the two Segal conditions are clearly equivalent and the two completeness conditions are equivalent because the maximal Kan complex contained in each $\mm{ev}_0(Y_n)=Y_{n, -, 0}$ is equivalent to the space $Y_{n, 0, 
-}$, by completeness of $Y_n$.
\end{proof}
\begin{remark}
It follows from \cite[Theorem 1.2.13]{lur09g} that a map $f\colon X\rt Y$ between Segal objects is a weak equivalence in the 2-categorical model structure if and only if it is a Segal equivalence, i.e.\ if $f$ is fully faithful and essentially surjective.
\end{remark}
The usual model structure for complete Segal spaces is a Bousfield localization of the 2-categorical model structure, whose fibrant objects are the complete Segal objects for which each $X_n$ is a Kan complex. Since $\mm{ev}_0\colon \Fun(\mb{\Delta}^\mm{op}, \sSet)\rt \sSet$ is a right Quillen equivalence between the Joyal model structure and the model structure for complete Segal spaces, this realizes the homotopy theory for $(\infty, 1)$-categories as a Bousfield localization of the homotopy theory for $(\infty, 2)$-categories. In particular, we can view each quasicategory $\cat{C}$ as an $(\infty, 2)$-category (which happens to be an $(\infty, 1)$-category) by means of its \emph{nerve} $\mc{N}(\cat{C})$ (levelwise equivalent to the object $t^!\cat{C}$ defined in \cite{joy06}, cf. \cite[Proposition 6.13]{cis10}). Recall that $\mc{N}(\cat{C})$ is the complete Segal space whose $n$-th space is given by the maximal Kan complex  $K(\mm{Fun}(\Delta[n], \cat{C}))$ contained in $\mm{Fun}(\Delta[n], \cat{C})$. This indeed gives a complete Segal space model for $\cat{C}$ since $\mm{ev}_0\mc{N}(\cat{C})$ is \emph{isomorphic} to $\cat{C}$. 

We think of the localization functor from Segal objects to Segal spaces as taking the free $(\infty, 1)$-category associated to an $(\infty, 2)$-category. The free Segal space associated to a Segal object $X$ can be described relatively easily, by group-completing all mapping categories of $X$:
\begin{lemma}\label{lem:free1caton2cat}
Let $X\colon \mb{\Delta}^\mm{op}\rt \sSet$ be a Segal object in the sense of Definition \ref{def:segalobject}. Then its levelwise Kan fibrant replacement, which we will denote by $|X|$, is (levelwise weakly equivalent to) a Segal space.
\end{lemma}
\begin{proof}
Since $X_0$ is already a Kan complex, we may assume that the map $X\rt |X|$ is the identity in degree $0$. Now observe that for any diagram $\cat{C}\rt \cat{E}\lt \cat{D}$ of quasicategories for which $\cat{E}$ is a Kan complex, the homotopy pullback in the Joyal model structure is equivalent to the homotopy pullback in the Kan-Quillen model structure: indeed, since the Kan-Quillen model structure is right proper, the natural comparison map between these two homotopy limits is provided by the map
$$\xymatrix{
\cat{C}\times\cat{D}\times_{\cat{E}\times\cat{E}} \Fun(N(J), \cat{E})\ar[r] & \cat{C}\times\cat{D}\times_{\cat{E}\times\cat{E}} \Fun(\Delta[1], \cat{E})
}$$
which is a trivial fibration since $\Fun(N(J), \cat{E})\rt \Fun(\Delta[1], \cat{E})$ is a trivial fibration for any Kan complex $\cat{E}$. The $n$-th Segal map of $|X|$ can then be identified with the composition
$$\xymatrix@C=0.8pc{
|X_n|\ar[r] & \big|X_1\times_{X_0}^{\scriptscriptstyle\mm{hJ}} \cdots \times_{X_0}^{\scriptscriptstyle\mm{hJ}} X_1\big|\ar[r] & \big|X_1\times_{X_0}^{\scriptscriptstyle\mm{hKQ}} \cdots \times_{X_0}^{\scriptscriptstyle\mm{hKQ}} X_1\big| \ar[r] & |X_1|\times^{\scriptscriptstyle\mm{hKQ}}_{X_0} \cdots \times_{X_0}^{\scriptscriptstyle\mm{hKQ}} |X_1|
}$$
where the second and third term are the Kan fibrant replacements of the relevant homotopy pullbacks in the Joyal, resp.\ the Kan-Quillen model structure. 
The first map is an equivalence by the original Segal condition on $X$, the second map is an equivalence by the above discussion and the last map is an equivalence since homotopy limits are invariant under weak equivalences.
\end{proof}

\subsection{The 2-category of spans}\label{sec:2catofspans}
We will now define our model for the $(\infty, 2)$-category $\mm{Span}(\cat{C}, W)$ whose mapping categories are the categories $\mm{Span}_{\cat{C}}^W(c, d)$ of Remark \ref{rem:relativespans}. This follows the construction in \cite{hau14} where $W=\cat{C}$.
\begin{construction}\label{con:spancategory}
For each $[n]\in \mb{\Delta}$, let $\mb{\Sigma}_n$ be the poset whose objects are pairs $(i, j)$ with $0\leq i\leq j\leq n$, where $(i, j)\leq (i', j')$ if $i\leq i'$ and $j\geq j'$. We will consider $\mb{\Sigma}_n$ as a relative category by declaring all maps $(i, j)\rt (i, j')$ to be weak equivalences. A map $\alpha\colon [m]\rt [n]$ induces a relative functor $\mb{\Sigma}_m\rt \mb{\Sigma}_n$ sending $(i, j)$ to $(\alpha(i), \alpha(j))$.

This yields a functor $N(\mb{\Sigma})\colon \mb{\Delta}\rt \cat{RelCat}\rt \sSet^+$ where the last functor is the marked nerve functor. 
For each $n$, let $\mm{Span}(\cat{C}, W)_n$ be the maximal sub-$\infty$-category of $\mm{Fun}(N(\mb{\Sigma}_n), \cat{C})$ whose
\begin{enumerate}
 \item[(a)] objects are relative functors $F\colon N(\mb{\Sigma}_n)\rt (\cat{C}, W)$ with the property that for any two tuples $(i, j)$ and $(i', j')$ with $i<i'\leq j'<j$, the corresponding square
 \begin{equation}\label{diag:pullbackspan}\vcenter{\xymatrix@rd@R=1.5pc@C=1.5pc{
 F(i, j) \ar[r]\ar[d]_-{\scriptscriptstyle\sim} & F(i', j)\ar[d]^-{\scriptscriptstyle\sim} \\
 F(i, j')\ar[r] & F(i', j')
 }}\end{equation}
 is cartesian in $\cat{C}$ (where the maps marked by $\sim$ are in $W$).
 \item[(b)] arrows are natural transformations $F\rt G$ of such relative functors, such that $F(i, i)\rt G(i, i)$ is an equivalence for all $0\leq i\leq n$.
\end{enumerate}
This results in a functor $\mm{Span}(\cat{C}, W)\colon \mb{\Delta}^\mm{op}\rt \sSet$ taking values in quasicategories.
\end{construction}
For example, an object in $\mm{Span}(\cat{C}, W)_1$ is a span
$$\xymatrix@R=1pc@C=1pc{
 & F(0, 1)\ar[ld]_{\scriptscriptstyle\sim}\ar[rd]\\
F(0, 0) & & F(1, 1)
}$$
whose left leg is contained in $W$ and an object in $\mm{Span}(\cat{C}, W)_2$ is a diagram of the form
$$\xymatrix@rd@R=1.5pc@C=1.5pc{
F(0, 2)\ar[d]_-{\scriptscriptstyle\sim}\ar[r] & F(1, 2)\ar[d]^-{\scriptscriptstyle\sim} \ar[r] & F(2, 2)\\
F(0, 1)\ar[d]_-{\scriptscriptstyle\sim}\ar[r] & F(1, 1)\\
F(0, 0)
}$$
where the square is cartesian.
\begin{lemma}\label{lem:lesssquares}
Let $\cat{C}$ be an $\infty$-category and let $F\colon N(\mb{\Sigma}_n)\rt \cat{C}$ be a functor. Then the following are equivalent:
\begin{enumerate}
\item for each $i<i'\leq j'<j$, the square \eqref{diag:pullbackspan} is cartesian in $\cat{C}$.
\item for each $j>i+1$, the square \eqref{diag:pullbackspan} is cartesian for $i'=i+1$ and $j'=j-1$.
\item $F$ is a right Kan extension of its restriction to the full subcategory $N(\mb{\Lambda}_n)\subseteq N(\mb{\Sigma}_n)$ consisting of $(i, j)$ with $j-i\leq 1$.
\end{enumerate}
\end{lemma}
\begin{proof}
The equivalence of (1) and (2) follows from the pasting lemma for cartesian squares in $\cat{C}$ \cite[Lemma 4.4.2.1]{lur09}. For the equivalence with (3), let $\mb{\Sigma}_n^{(k)}$ be the full subcategory of $\mb{\Sigma}_n$ on those pairs $(i, j)$ for which $j-i\leq k$, where $1\leq k \leq n$. Note that $\mb{\Lambda}_n=\mb{\Sigma}_n^{(1)}$. For fixed $(i, j)$ with $j-i=k+1$, the obvious inclusion
$$\xymatrix{
\Big\{(i, j-1)\rt (i+1, j-1)\lt (i+1, j)\Big\} \ar[r] & (i, j)/\mb{\Sigma}_n^{(k)}
}$$
is (homotopically) final. This means that $F\big|N(\mb{\Sigma}_n^{(k+1)})$ is a right Kan extension of $F\big|N(\mb{\Sigma}_n^{(k)})$ if and only if the square \eqref{diag:pullbackspan} is cartesian for $i<i+1\leq j-1<j$ for all $(i, j)$ such that $j-i=k+1$. Using \cite[Proposition 4.3.2.8]{lur09}, it follows that condition (2) is equivalent to condition (3).
\end{proof}
\begin{lemma}[\cite{hau14}]\label{lem:2catofspans}
The simplicial quasicategory $\mm{Span}(\cat{C}, W)\colon \mb{\Delta}^\mm{op}\rt \sSet$ is a Segal object in the sense of Definition \ref{def:segalobject}.
\end{lemma}
\begin{proof}
To see that $\mm{Span}(\cat{C}, W)$ is Reedy fibrant, let $\mm{Span}(\cat{C}, W)_n\rt M_n\mm{Span}(\cat{C}, W)$ be its $n$-th matching map for $n\geq 1$ (the case $n=0$ is trivial). 
This matching map fits into a commuting diagram
$$\xymatrix{
\mm{Span}(\cat{C}, W)_n\ar[r]\ar[d]_{\subseteq} & M_n\mm{Span}(\cat{C}, W)\ar[d]^\subseteq\\
\mm{Fun}(N(\mb{\Sigma}_n), \cat{C})\ar[r] & \mm{Fun}(K, \cat{C})
}$$
where $K\subseteq N(\mb{\Sigma}_n)$ is the union of all simplicial subsets $N(\mb{\Sigma}_{n-1})\subseteq N(\mb{\Sigma}_n)$ induced by face maps. The bottom map is a categorical fibration between quasicategories since the Joyal model structure is cartesian closed. The vertical maps are inclusions of maximal sub-$\infty$-categories containing certain vertices and arrows. This immediately implies that the matching map is an inner fibration. In fact, since $K\subseteq N(\mb{\Sigma}_n)$ contains all vertices of the form $(i, i)$, it follows that an arrow in $\mm{Fun}(N(\mb{\Sigma}_n), \cat{C})$ is contained in $\mm{Span}(\cat{C}, W)_n$ if and only if its image in $\mm{Fun}(K, \cat{C})$ is contained in $M_n\mm{Span}(\cat{C}, W)$. This implies that the top horizontal functor has the right lifting property against $\{0\}\rt J$ as well.

To see that $\mm{Span}(\cat{C}, W)$ satisfies the Segal conditions, observe that the $n$-th Segal map fits into a commuting square
$$\xymatrix@!C=6.2pc{
\mm{Span}(\cat{C}, W)_n \ar[rr]\ar[d]_\subseteq & & \mm{Span}(\cat{C}, W)_1\times_{\mm{Span}(\cat{C}, W)_0} \cdots \times_{\mm{Span}(\cat{C}, W)_0} \mm{Span}(\cat{C}, W)_1\ar[d]^\subseteq \\
\Fun^\mm{Ran}(N(\mb{\Sigma}_n), \cat{C}) \ar[r]_-{\subseteq} & \Fun(N(\mb{\Sigma}_n), \cat{C})\ar[r] & \Fun(N(\mb{\Lambda}_n), \cat{C})
 }$$ 
 where $\mb{\Lambda}_n$ is as in Lemma \ref{lem:lesssquares} and $\Fun^\mm{Ran}(N(\mb{\Sigma}_n), \cat{C})$ is the full subcategory of those functors that are right Kan extensions of their restrictions to $\Lambda_n$. We now make the following two observations:
 \begin{itemize}
 \item a right Kan extension $F\colon N(\mb{\Sigma}_n)\rt \cat{C}$ is a relative functor (with respect to the weak equivalences of Construction \ref{con:spancategory}) if and only if its restriction to $\mb{\Lambda}_n$ is a relative functor, since hypercovers in $\cat{C}$ are stable under base change. This means precisely that its restriction lies in the image of the right vertical functor.
 \item a natural transformation $F\rt G$ in $\Fun^\mm{Ran}(N(\mb{\Sigma}_n), \cat{C})$ induces equivalences $F(i, i)\rt G(i, i)$ if and only if its restriction to $\mb{\Lambda}_n$ does.
 \end{itemize}
 It follows from these two observations that the above square is cartesian. The right vertical functor takes values in the full subcategory of $\Fun(N(\mb{\Lambda}_n), \cat{C})$ on those diagrams that \emph{admit} right Kan extensions, by Lemma \ref{lem:lesssquares} and \cite[Lemma 4.3.2.13]{lur09}, together with the fact that $\cat{C}$ admits pullbacks along hypercovers. It now follows from \cite[Proposition 4.3.2.15]{lur09} that the Segal map is a trivial fibration.
\end{proof}
The simplicial quasicategory $\mm{Span}(\cat{C}, W)$ is our model for the $(\infty, 2)$-category of spans in $\cat{C}$ whose `domain leg' is a hypercover. Note that for two objects $c, d\in \mm{Span}(\cat{C}, W)_0=K(\cat{C})$, the quasi-category of maps between them is given by the quasicategory $\mm{Span}^W_{\cat{C}}(c, d)$ of Remark \ref{rem:relativespans}. By Lemma \ref{lem:free1caton2cat}, the free $(\infty, 1)$-category $|\mm{Span}(\cat{C}, W)|$ generated by the $(\infty, 2)$-category $\mm{Span}(\cat{C}, W)$ can be obtained simply by group-completing all these mapping categories. 

Let us now turn to proving that the Segal space $|\mm{Span}(\cat{C}, W)|$ is indeed a Segal space model for the localization $\cat{C}[W^{-1}]$. For each $n\geq 0$, let $\phi_n\colon \mb{\Sigma}_n\rt [n]$ be the functor sending $(i, j)$ to $i$ and observe that these functors together yield a natural tranformation of $\mb{\Delta}$-indexed diagrams in $\cat{Cat}$. For each $n$, precomposing with $\phi_n$ gives a functor
$$\xymatrix{
\phi_n^*\colon \Fun(\Delta[n], \cat{C})\ar[r] & \Fun(N(\mb{\Sigma}_n), \cat{C})
}$$
sending a sequence $c_0\rt \cdots \rt c_n$ to the diagram
\begin{equation}\label{diag:reducedspans}\vcenter{
\xymatrix@R=1.2pc@C=1.2pc@rd{
c_0\ar@{=}[d]\ar[r] & c_1\ar@{=}[d]\ar[r] & c_2\ar@{=}[d]\ar[r] & \cdots \ar@{=}[d]\ar[r] & c_n.\\
\cdots\ar@{=}[d]\ar[r] & \cdots \ar[r]\ar@{=}[d] & \cdots \ar@{=}[d]\ar[r] & \cdots\\
c_0\ar@{=}[d]\ar[r] & c_1\ar@{=}[d]\ar[r] & c_2\\
c_0\ar@{=}[d]\ar[r] & c_1\\
c_0
}}\end{equation}
Such a diagram is clearly an object in $\mm{Span}(\cat{C}, W)_n$ and it follows that restriction along $\phi$ induces a natural functor
$$\xymatrix{
\phi^*_n\colon \mc{N}(\cat{C})_n=K(\Fun(\Delta[n], \cat{C}))\ar[r] & \mm{Span}(\cat{C}, W)_n.
}$$
Together these functors determine a map of bisimplicial sets $u\colon \mc{N}(\cat{C})\rt \mm{Span}(\cat{C}, W)$.
\begin{theorem}\label{thm:bicatofspans}
Let $(\cat{C}, W)$ be an $\infty$-category with hypercovers. The composite map 
$$\xymatrix{
f\colon \mc{N}(\cat{C})\ar[r]^-u & \mm{Span}(\cat{C}, W)\ar[r] & |\mm{Span}(\cat{C}, W)|
}$$
realizes the latter as a Segal space model for the localization $\cat{C}[W^{-1}]$.
\end{theorem}
Let us first prove the theorem in a special case:
\begin{lemma}
Let $\cat{C}$ be an $\infty$-category, considered as an $\infty$-category with hypercovers $(\cat{C}, \mm{eq})$ given by the equivalences. Then the map $u\colon \mc{N}(\cat{C})\rt \mm{Span}(\cat{C}, \mm{eq})$ is an equivalence of Segal spaces.
\end{lemma}
\begin{proof}
Note that each $\mm{Span}(\cat{C}, \mm{eq})_n$ is a Kan complex, since a natural transformation $F\rt G$ of diagrams of the form
$$\xymatrix@rd@R=1.5pc@C=1.5pc{
\cdots\ar[d]_-{\scriptscriptstyle\simeq}\ar[r] & \cdots\ar[d]^-{\scriptscriptstyle\simeq} \ar[r] & \cdots\\
F(0, 1)\ar[d]_-{\scriptscriptstyle\simeq}\ar[r] & F(1, 1)\\
F(0, 0)
}$$
is a natural equivalence if and only if each $F(i, i)\rt G(i, i)$ is an equivalence by the 2-out-of-3 property. It follows that $\mm{Span}(\cat{C}, \mm{eq})_n\subseteq K(\Fun(\mb{\Sigma}_n, \cat{C}))$ is given by the full subcomplex consisting of those diagrams $F\colon \mb{\Sigma}_n\rt \cat{C}$ for which the maps $F(i, j)\rt F(i, j')$ are equivalences. Note that this is equivalent to asking that $F$ is a left Kan extension of its restriction to the full subcategory
$$\xymatrix{
j\colon \Big\{(0,n)\rt (1, n)\rt ... \rt (n, n)\Big\}\ar[r] & \mb{\Sigma}_n.
}$$
It follows that restriction along $j$ induces a trivial fibration $j^*\colon \mm{Span}(\cat{C}, \mm{eq})_n\rt K(\Fun(\Delta[n], \cat{C}))$, of which the map $u_n\colon \mc{N}(\cat{C})_n\rt \mm{Span}(\cat{C}, \mm{eq})$ is a section.
\end{proof}
\begin{proof}[Proof (of Theorem \ref{thm:bicatofspans})]
The map of relative $\infty$-categories $(\cat{C}, W)\rt (\cat{C}[W^{-1}], \mm{eq})$ induces a diagram of Segal objects
$$\xymatrix{
\mc{N}(\cat{C})\ar[r]\ar[d] & \mm{Span}(\cat{C}, W)\ar[r] \ar[d] & |\mm{Span}(\cat{C}, W)|\ar[d]\\
\mc{N}(\cat{C}[W^{-1}])\ar[r]_-\simeq & \mm{Span}(\cat{C}[W^{-1}], \mm{eq})\ar[r]_-\simeq & |\mm{Span}(\cat{C}[W^{-1}], \mm{eq})|
}$$
in which the bottom row is given by levelwise weak equivalences. To prove the theorem, it suffices to show that the right vertical map is an equivalence of Segal spaces. Since the map in degree $0$ is weakly equivalent to the map $K(\cat{C})\rt K(\cat{C}[W^{-1}])$, which is clearly surjective on connected components, it suffices to prove that the right vertical map is a fully faithful map of Segal spaces. Equivalently, it suffices to show that the middle vertical map is fully faithful in the sense that on mapping categories, it is given by a Kan-Quillen equivalence between quasicategories.

But for fixed objects $c, d\in \mc{N}(\cat{C})_0 = K(\cat{C})$, the value of the middle vertical functor is given on the mapping space from $c$ to $d$ by the map of quasicategories of spans
$$\xymatrix{
\mm{Span}_{\cat{C}}^W(c, d)\ar[r] & \mm{Span}_{\cat{C}[W^{-1}]}^\mm{eq}(c, d)
}$$ 
from Remark \ref{rem:relativespans}. This map is a Kan-Quillen equivalence by Corollary \ref{cor:fullyfaithfulonhypercoverlocalization}.
\end{proof}
Let us conclude with some simple properties of the functor $u\colon \mc{N}(\cat{C})\rt \mm{Span}(\cat{C}, W)$ to the category of spans itself, rather than its associated $(\infty, 1)$-category.
\begin{lemma}
The map $u\colon \mc{N}(\cat{C})\rt \mm{Span}(\cat{C}, W)$ sends all hypercovers $f\colon d\rt c$ in $\mc{N}(\cat{C})$ to morphisms in the $(\infty, 2)$-category $\mm{Span}(\cat{C}, W)$ that have right adjoints.
\end{lemma}
\begin{proof}
In light of \cite{rie13} (see also \cite{hau14}), an arrow in an $(\infty, 2)$-category has a right adjoint if and only if it has a right adjoint at the level of the homotopy $2$-category. Given a morphism $f\colon d\rt c$ in $W$, let $d \stackrel{=}{\lt} d\rt c$ be the associated span in $\mm{Span}(\cat{C}, W)$. A right adjoint to this morphism in $\ho_2(\mm{Span}(\cat{C}, W))$ is provided by the span $c \lt d\stackrel{=}{\rt} d$. Indeed, there are obvious unit and counit maps
$$\xymatrix@R=0.7pc{
& d\ar@/_0.7pc/[ldd]_=\ar@/^0.7pc/[rdd]^=\ar[d]  & & &  & d\ar@/_0.7pc/[ldd]_f\ar@/^0.7pc/[rdd]^f\ar[d]_(0.6)f \\
& d\times_c d\ar[ld]^\sim\ar[rd] & & & & c\ar[ld]^=\ar[rd]_=\\\
d & & d & & c & & c
}$$
arising from the diagonal map (as well as the map $f$ itself) in the homotopy category of $\mm{Span}(\cat{C}, W)_1$.
\end{proof}
Each commuting diagram in the homotopy 2-category $\ho_2(\cat{C})$
\begin{equation}\label{diag:pullbackofhypercovers}\vcenter{\xymatrix{
d'\ar[r]^{f'}\ar[d]_{g'} & c'\ar[d]^g\\
d\ar[r]_f & c
}}\end{equation}
induces a homotopy commuting diagram in the 2-category $\ho_2(\mm{Span}(\cat{C}, W))$. When $f$ and $f'$ are hypercovers, their images in the homotopy 2-category $\ho_2(\mm{Span}(\cat{C}, W))$ have adjoints $f^\perp$ and $f'^\perp$, and the invertible 2-cell $g\circ f'\simeq g'\circ f$ induces a $2$-cell $g'\circ f'^\perp \rt f^\perp \circ g$ in $\ho_2(\mm{Span}(\cat{C}, W))$.
\begin{lemma}
Consider a cartesian square  in $\cat{C}$ of the form \eqref{diag:pullbackofhypercovers}, in which $f$ and $f'$ are hypercovers. Then the induced 2-cell $g'\circ f'^\perp\rt f^\perp\circ g$ in $\ho_2(\mm{Span}(\cat{C}, W))$ is invertible.
\end{lemma}
\begin{proof}
The induced 2-cell can be identified with the dotted map
$$\xymatrix@R=1pc{
 & & d'\ar[llddd]\ar[rrddd] \ar@{..>}[d]\\
 & & x \ar[ld]\ar[rd]\\
 & c'\ar[ld]^= \ar[rd]_g & & d\ar[ld]^f \ar[rd]_=\\
c' & & c & & d
}$$
from $d'$ into the pullback $x$ in $\cat{C}$, which is an equivalence by assumption.
\end{proof}
\begin{remark}
It follows from the previous two results that a map of $(\infty, 2)$-categories $f\colon \cat{C}\rt \cat{D}$ can only factor over $\cat{C}\rt \mm{Span}(\cat{C}, W)$ if it sends all arrows in $W$ to arrows in $\cat{D}$ with right adjoints, such that for any cartesian square \eqref{diag:pullbackofhypercovers}, the corresponding Beck-Chevalley map in $\ho_2(\cat{D})$ is invertible. Work of Gaitsgory and Rozenblyum \cite{gai16} indicates that the functor $\cat{C}\rt \mm{Span}(\cat{C}, W)$ is universal among such maps of $(\infty, 2)$-categories. This universal property would imply the universal property of $|\mm{Span}(\cat{C}, W)|$ as the localization of $\cat{C}$ at $W$ by abstract nonsense, since in an $(\infty, 1)$-category, a left adjoint map is precisely an equivalence and all Beck-Chevalley maps are (trivially) invertible.
\end{remark}

\bibliographystyle{abbrv}
\bibliography{bibliography_groupoids}

%

\end{document}